\documentclass[a4paper,12pt]{article}
\usepackage{amssymb}
\usepackage{amsthm}
\usepackage{tikz}
\usepackage{tikz,pgfplots}
\usepackage{amsmath,amssymb}
\usepackage{todonotes}
\usepackage{url}
\usepackage{diagbox}

\usepackage{color}
\usepackage{graphicx}
\usepackage{placeins,caption}

\DeclareMathOperator{\gp}{gp}
\DeclareMathOperator{\mono }{mp}

\DeclareMathOperator{\g}{g}
\DeclareMathOperator{\cp}{\,\square\,}
\usepackage[top=3cm,bottom=3cm,left=2.8cm,right=2.8cm]{geometry}
\newtheorem{theorem}{Theorem}[section]
\newtheorem{lemma}[theorem]{Lemma}

\newtheorem{proposition}[theorem]{Proposition}

\theoremstyle{definition}

\newtheorem{definition}[theorem]{Definition}

\newcommand{\address}[1]{#1}
\begin{document}
	
	\title{Lower General Position Sets in Graphs}
	\author{Gabriele Di Stefano $^{a}$ \\ \texttt{\footnotesize gabriele.distefano@univaq.it}  
		\and
		Sandi Klav\v{z}ar $^{b,c,d}$ \\ \texttt{\footnotesize sandi.klavzar@fmf.uni-lj.si}  
		\and
		Aditi Krishnakumar $^{e}$ \\ \texttt{\footnotesize aditikrishnakumar@gmail.com}  
		\and
		James Tuite $^{e}$ \\ \texttt{\footnotesize james.t.tuite@open.ac.uk} 
		\and
		Ismael Yero $^{f}$ \\ \texttt{\footnotesize ismael.gonzalez@uca.es} 
	}
	
	\maketitle
	
	\address{
		$^a$ Department of Information Engineering, Computer Science and Mathematics, University of L'Aquila,  Italy
		
		$^b$ Faculty of Mathematics and Physics, University of Ljubljana, Slovenia
		
		$^c$ Institute of Mathematics, Physics and Mechanics, Ljubljana, Slovenia
		
		$^d$ Faculty of Natural Sciences and Mathematics, University of Maribor, Slovenia
		
		$^e$ Department of Mathematics and Statistics, Open University, Milton Keynes, UK
		
		$^f$ Departamento de Matem\'aticas, Universidad de C\'adiz, Algeciras, Spain
		
	}

	\begin{abstract}
		A subset $S$ of vertices of a graph $G$ is a \emph{general position set} if no shortest path in $G$ contains three or more vertices of $S$. In this paper, we generalise a problem of M. Gardner to graph theory by introducing the \emph{lower general position number} $\gp ^-(G)$ of $G$, which is the number of vertices in a smallest maximal general position set of $G$. We show that ${\rm gp}^-(G) = 2$ if and only if $G$ contains a universal line and determine this number for several classes of graphs, including Kneser graphs $K(n,2)$, line graphs of complete graphs, and Cartesian and direct products of two complete graphs. We also prove several realisation results involving the lower general position number, the general position number and the geodetic number, and compare it with the lower version of the monophonic position number. We provide a sharp upper bound on the size of graphs with given lower general position number. Finally we demonstrate that the decision version of the lower general position problem is NP-complete. 
	\end{abstract}
	
	
	\noindent
	AMS Subj.\ Class.\ (2020): 05C12, 05C69, 68Q25

	\section{Introduction}\label{sec:intro}
	
	The \emph{general position problem} originated in a puzzle by Dudeney in his book~\cite{dudeney-1917}. It can be stated as follows: what is the largest number of pawns that can be placed on an $n \times n$ chessboard such that no three pawns lie on a straight line? This geometrical problem is also known as the \emph{no-three-in-line problem}. An obvious upper bound is $2n$, which is achieved for $n \leq 46$. However, for larger $n$ the problem remains open. Erd\H{o}s showed how to place $n-o(n)$ pawns on the chessboard with no three in line (his proof is recorded in the paper~\cite{Roth} by Roth), which was subsequently improved by Hall et al. to $\frac{3n}{2}-o(n)$ pawns~\cite{Hall}. It is conjectured that the true answer is $\frac{\pi n}{\sqrt{3}}-o(n)$ in~\cite{Guy} (see~\cite{Pegg} for a correction). The problem has been called ``one of the oldest and most extensively studied geometric questions concerning lattice points''~\cite{BrassMoserPach}.
	
	The general position problem was generalised to the setting of graph theory independently in~\cite{ullas-2016, Korner, manuel-2018} as follows.
	
	\begin{definition}
		A set $S \subseteq V(G)$ is in \emph{general position} if no shortest path in $G$ contains three or more vertices of $S$; such a set is a \emph{general position set}. The \emph{general position number}  $\gp(G)$ of $G$ is the number of vertices in a largest general position set. The \emph{general position problem} asks for a largest general position set in a given graph. 
	\end{definition}
	
	Papers in this very active field of research include~\cite{Ghorbani-2021, Klavzar-2019, Neethu-2020, Patkos-2019, Thomas-2020, tian-2021, yao-2022}. Several variations of the problem have been considered, including using the Steiner distance instead of the regular graph distance~\cite{KlaKuzPetYer}, or confining attention to shortest paths of bounded length~\cite{KlaRalYer}. Games involving general position sets have also been treated in~\cite{KlaSam} and~\cite{KlaNeeCha}, a dynamic variant of the problem was considered in~\cite{KlaKriTuiYer} and a local version of general position sets was studied in~\cite{ThaChaTuiThoSteErs}.

	The edge version of the general position problem has also been recently studied in~\cite{manuel-2022}. A related problem is the \emph{monophonic position problem} obtained by replacing ``shortest path'' in the general position problem by ``induced path'', see~\cite{thomas-2023}. Another variant of the general position problem is the \emph{mutual-visibility problem} that asks for a largest set of vertices $S$, such that for each pair of vertices in $S$ there is a shortest path connecting them that does not contain a third vertex of $S$, see~\cite{DiStefano-2022}.

	A new slant on this old problem was given by Martin Gardner (the modern day Dudeney), who asked the following question in his column in Scientific American: ``Instead of asking for the maximum number of counters that can be put on an order-$n$ board, no three in line, let us ask for the minimum that can be placed such that adding one more counter on any vacant cell will produce three in line''~\cite{Gar}. If a greedy algorithm is used to produce a general position set, then the answer to Gardner's problem represents the worst-case output. This problem was treated in~\cite{AdeHolKel, AicEppHai, CooPikSchWar}. The most recent of these articles~\cite{AicEppHai} refers to this problem as the \emph{geometric domination problem} and gives a lower bound of $\Omega (n^{\frac{2}{3}})$ and an upper bound of $2 \left \lceil \frac{n}{2} \right \rceil $ for an $n \times n$ grid.
	
	In this paper we extend Gardner's problem to graph theory by asking for the smallest maximal general position sets, i.e.\ the smallest general position sets that cannot be extended without creating three in a line.
	
	\begin{definition}
		A general position set $S$ in a graph $G$ is \emph{maximal} if there is no general position set of $G$ containing $S$ as a proper subset. The \emph{lower general position number} $\gp^-(G)$ of $G$ is the number of vertices in a smallest maximal general position set of $G$, also called a \emph{lower general position set}.
	\end{definition}
	
	An example of these concepts in the context of the Petersen graph $P$ can be seen in Fig.~\ref{fig:Petersen}. From~\cite{manuel-2018} we know that $\gp(P) = 6$; the set of white vertices in the figure represents a largest general position set, whilst the grey vertices form a smallest maximal general position set, so that $\gp^-(P) = 4$.
	
	Since any vertex of a (non-trivial) graph $G$ is not a maximal general position set, for any graph $G$ we have:
	\begin{equation}\label{eq-trivial bounds}
		2\le \gp^-(G)\le \gp(G).
	\end{equation}
	It is not difficult to see that for any tree $T$ (or, more generally, any graph with a bridge) we have $\gp^-(T)=2$, so it is possible to have equality with the lower bound in Inequality~\ref{eq-trivial bounds}. In Section \ref{sec:gp-=2} we will take a closer look at the graphs $G$ which satisfy $\gp^-(G)=2$. The upper bound in Inequality~\ref{eq-trivial bounds} is also tight, since any complete graph $K_n$ satisfies $\gp^-(K_n) = n = \gp(K_n)$.
	
	\begin{figure}[ht!]
		\centering
		\begin{tikzpicture}[x=0.2mm,y=-0.2mm,inner sep=0.2mm,scale=0.7,thick,vertex/.style={circle,draw,minimum size=10}]
			\node at (180,200) [vertex,fill=white] (v1) {};
			\node at (8.8,324.4) [vertex,fill=gray] (v2) {};
			\node at (74.2,525.6) [vertex,fill=white] (v3) {};
			\node at (285.8,525.6) [vertex,fill=white] (v4) {};
			\node at (351.2,324.4) [vertex,fill=gray] (v5) {};
			\node at (180,272) [vertex,fill=white] (v6) {};
			\node at (116.5,467.4) [vertex,fill=gray] (v7) {};
			\node at (282.7,346.6) [vertex,fill=white] (v8) {};
			\node at (77.3,346.6) [vertex,fill=white] (v9) {};
			\node at (243.5,467.4) [vertex,fill=gray] (v10) {};

			\path
			(v1) edge (v2)
			(v1) edge (v5)
			(v2) edge (v3)
			(v3) edge (v4)
			(v4) edge (v5)
			
			(v6) edge (v7)
			(v6) edge (v10)
			(v7) edge (v8)
			(v8) edge (v9)
			(v9) edge (v10)
			
			(v1) edge (v6)
			(v2) edge (v9)
			(v3) edge (v7)
			(v4) edge (v10)
			(v5) edge (v8)

			;
		\end{tikzpicture}
		\caption{The Petersen graph with a maximum general position set (white) and a lower general position set (grey)}
		\label{fig:Petersen}
	\end{figure}
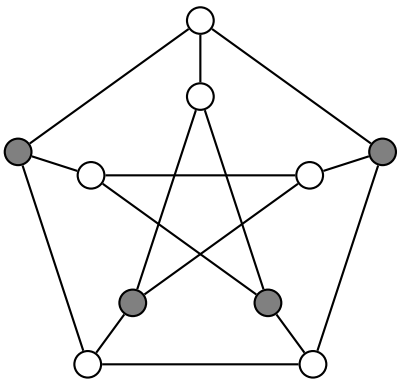
	We now provide some necessary definitions. All graphs considered in this paper are finite, undirected, simple and connected. The order and the size of a graph $G$ will be denoted by $n(G)$ and $m(G)$, respectively. We will write $u \sim v$ if vertices $u$ and $v$ are adjacent. The subgraph of $G$ induced by a subset $S \subseteq V(G)$ will be denoted by $G[S]$. The \emph{distance} $d_G(u,v)$ between vertices $u$ and $v$ of a graph $G$ is the length of a shortest $u,v$-path. The \emph{interval} $I[u,v]$ of $u$ and $v$ is the set of vertices that lie on at least one shortest $u,v$-path. The \emph{geodetic closure} $I[S]$ of a set $S \subseteq V(G)$ is the union $\bigcup _{u,v \in S}I[u,v]$ and $S$ is \emph{geodetic} if $I[S] = V(G)$. The \emph{geodetic number} $\g (G)$ of $G$ is the number of vertices in a smallest geodetic set of $G$, see~\cite{ullas-2022}.
	
	In connection with this last parameter, note that a geodetic set $S$ in a graph $G$ has the property that no vertex can be added to $S$ without creating three in a line (although there is no guarantee that $S$ is in itself in general position). Thus, if a geodetic set is in general position, then it is also a maximal general position set. This may suggest that a graph with a small geodetic number might also have a small lower general position number.
	
	The \emph{clique number} $\omega (G)$ is the number of vertices in a largest clique (i.e. set of mutually adjacent vertices) of $G$, whilst $\omega^-(G)$ will represent the number of vertices in a smallest maximal clique of $G$. A subset $S \subseteq V(G)$ is an \emph{independent union of cliques} if each component of $G[S]$ is a clique; the number of vertices in a largest independent union of cliques is the \emph{independent union of cliques number} $\alpha ^{\omega }(G)$, whereas we will denote the number of vertices in a smallest maximal independent union of at least two cliques by $\alpha ^{\omega ^-}(G)$ (for convenience we do not count a clique as an independent union of cliques here). The join of graphs $G$ and $H$ will be denoted by $G \vee H$. We also adopt the convention that $[n]$ stands for the set $\{ 1,\ldots ,n\}$.
	
	The plan of this paper is as follows. In Section~\ref{sec:gp-=2} we consider graphs having lower general position number equal to two. In particular, we show that these graphs coincide with the graphs that contain a universal line. In Section~\ref{sec:realisations} we prove several realisation results involving the lower general position number, the general position number and the geodetic number. We also give an upper bound on the size of a graph with given lower general position number and characterise the graphs which attain equality. In Section~\ref{sec:comp} we prove that the decision version of the lower general position problem is NP-complete. In Section~\ref{sec:families} we determine the lower general position number of the Kneser graphs $K(n,2)$, the line graphs of complete graphs, and the Cartesian and direct products of two complete graphs. In Section~\ref{sec:lower mp number} we relate the lower general position number to the lower monophonic position number. Finally in Section~\ref{sec:conclusion} we suggest several avenues for further research.

	\section{Graphs $G$ with $\gp^-(G) = 2$}
	\label{sec:gp-=2}
	
	Based on the trivial lower bound in Inequality \eqref{eq-trivial bounds}, it is natural (as with many other graph parameters) to consider characterising the family of graphs achieving equality. However, as we now show, this appears to be a very challenging problem. To see this, let us first invoke the following concept from the theory of metric spaces.
	Let $M = (X, d_M)$ be an arbitrary metric space and $x,y\in X$. Then the {\em line} ${\cal L}_M(x,y)$ induced by $x$ and $y$ is the following set of points from $M$:
	$$\{z\in X:\ d_M(x,y) = d_M(x,z) + d_M(z,y)\ {\rm or}\ d_M(x,y) = |d_M(x,z) - d_M(z,y)|\}\,.$$
	The line ${\cal L}_M(x,y)$ is {\em universal} if it contains the whole set $X$. Considering a graph $G$ as a metric space, these definitions transfer directly to $G$.
	
	Let $\ell(M)$ denote the number of distinct lines in $M$. Chen and Chv\'atal~\cite{ChenChv} conjectured that if $\ell(M) < |X|$, then $M$ has a universal line. The problem remains open; a summary of what is known about it up to 2018 can be found in~\cite{chvatal-2018}. For the case of graphs, the {C}hen-{C}hv\'{a}tal Conjecture has been verified in particular for a certain class of graphs containing both chordal graphs and distance-hereditary graphs~\cite{aboulker-2018}, for $(q,q - 4)$-graphs~\cite{schrader-2020}, and for graphs with no induced house or induced cycle of length at least $5$~\cite{aboulker-2022}.
	
	Closely related to the {C}hen-{C}hv\'{a}tal Conjecture is an open problem~\cite[Problem 1.2]{rodriguez-2022} to determine necessary and sufficient conditions for a graph to have a universal line. It turns out that the existence of a universal line in $G$ is equivalent to $\gp^-(G) = 2$.
	
	\begin{proposition}\label{prop:gp-=2}
		Let $G$ be a graph. Then $\gp^-(G) = 2$ if and only if $G$ has a universal line.
	\end{proposition}
	\begin{proof}
		Suppose first that $\gp^-(G) = 2$ and let $\{u,v\}$ be a maximal general position set of $G$. Then for every vertex $w\in V(G)\setminus \{u,v\}$, the triple of vertices $u,v,w$ lies on a shortest path. That is, one of the following equations applies: $d_G(u,v) = d_G(u,w) + d_G(w,v)$, $d_G(u,w) = d_G(u,v) + d_G(v,w)$ or $d_G(v,w) = d_G(v,u) + d_G(u,w)$. This in turn implies that ${\cal L}_G(u,v) = V(G)$, hence $G$ has a universal line. Conversely, if $G$ has a universal line ${\cal L}_G(u,v)$, then by the same argument, $\{u,v\}$ is a maximal general position set of $G$.
	\end{proof}
	
	Therefore the problem of the existence of universal lines considered in~\cite{rodriguez-2022} is equivalent to efficiently characterising connected graphs with lower general position number $2$. Simple sufficient conditions for $\gp^-(G) = 2$ to hold are: (i) having a bridge; (ii) $\g(G) = 2$; (iii) being bipartite, see~\cite{beaudou-2015}. If $G$ is a block graph, then $\gp^-(G) = 2$ if and only if $K_2$ is a block of $G$, see~\cite[Corollary 3.4]{rodriguez-2022}. In the same paper, Rodr\'{\i}guez-Vel\'{a}zquez characterised Cartesian product graphs having a universal line. (For the definition of the Cartesian product and its properties we refer to the book~\cite{ikr}.) In view of Proposition~\ref{prop:gp-=2} his result can be reformulated as follows.
	
	\begin{theorem} {\rm \cite[Theorem 4.1]{rodriguez-2022}}
		\label{thm:cp}
		Let $G$ and $H$ be non-trivial, connected graphs. Then $\gp^-(G\cp H) = 2$ if and only if one of the following conditions holds.
		\begin{enumerate}
			\item[{\em (i)}] $G$ or $H$ has a maximal general position set consisting of two adjacent vertices.
			\item[{\em (ii)}] $\g(G) = 2$ and $\g(H) = 2$.
		\end{enumerate}
	\end{theorem}
	
	Other graphs with lower general position number $2$ are: (i) graphs obtained from any graph $G$ such that $\gp^-(G) = 2$ and any graph $H$ by fixing a vertex $u \in V(G)$ and joining it to at least one vertex in every component of $H$ and (ii) any complete multipartite graph with a part of cardinality $2$, cf.\ Proposition~\ref{prop:cycle+multipartite}(ii).
	
	To conclude this section we determine the cycles and complete multipartite graphs that have lower general position number equal to $2$.
	
	\begin{proposition}\label{prop:cycle+multipartite}
		(i) If $n\ge 3$, then
		\[ \gp^- (C_n)=\begin{cases}
			2; &\text{$n$ even, } \\
			3; & \text{$n$ odd.}
		\end{cases}
		\]
		(ii) If $t\ge 2$ and $r_1 \geq \dots \geq r_t \geq 2$, then \[ \gp^-(K_{r_1,\dots ,r_t}) = \min \{ t,r_t\} .\]
	\end{proposition}
	
	\begin{proof}
		(i) If $n$ is even, then $C_n$ is bipartite and hence $\gp^- (C_n) = 2$. Now let $n$ be odd. Identify the vertex set of $C_n$ with $\mathbb{Z}_n$ in the natural manner. As $\gp(C_n) = 3$ for $n = 3$ and $n \geq 5$, we have $2 \leq \gp^-(G) \leq 3$ and the result will follow if we show that any set of two vertices of $C_n$ can be extended to form a general position set of three vertices. Without loss of generality, let this set be $S = \{ 0,i\} $, where $i < n-i$. Then it is easily checked that if $i$ is odd, then the set $\{ 0,i,\frac{n+i}{2}\} $ is in general position and if $i$ is even, then $\{ 0,i,\frac{n+i-1}{2}\} $ is in general position.
		
		(ii) Each of the partite sets of the complete multipartite graph is a maximal general position set, so it follows that $\gp^- (K_{r_1,\dots ,r_t}) \leq r_t$. Suppose that $S$ is a maximal general position set containing vertices from different partite sets. Then $S$ can contain at most one vertex from each of the $t$ parts; hence $S$ induces a clique and by maximality $S$ must contain $t$ vertices, one from each part. Hence $\gp^-(K_{r_1,\dots ,r_t}) = \min \{ t,r_t\} $.
	\end{proof}

	\section{Realisation results}
	\label{sec:realisations}
	
	In this section we explore the relationship between the lower general position number, the general position number and the geodetic number. We begin by proving a realisation result for the lower general position number and geodetic number of a graph. Recall that it was suggested in Section~\ref{sec:intro} that a graph with small geodetic number might be expected to have a small maximal geodetic set; although this intuition turns out to be true for graphs with geodetic number two or three, we now show that in general this is not the case. To this end, we begin by determining the lower general position number of the join of two graphs. Recall our convention that when finding the smallest number of vertices in a maximal independent union of cliques $\alpha ^{\omega ^-}(G)$ we do not allow unions consisting of one clique.
	
	\begin{lemma}\label{lem:lower gp of join}
		If $G$ and $H$ are graphs, then
		\[ \gp^-(G \vee H) = \min \{ \omega ^-(G)+\omega ^-(H),\alpha ^{\omega ^-}(G),\alpha ^{\omega ^-}(H)\} . \]
	\end{lemma}
	
	\begin{proof}
		Let $S$ be a smallest maximal general position set of $G \vee H$. Suppose that $S$ contains vertices of both $G$ and $H$; then both $S \cap V(G)$ and $S \cap V(H)$ must induce cliques, since if $u,u' \in V(G), v \in V(H)$ and $u \not \sim u'$, then $u,v,u'$ would be a path of length two contained in $S$. Conversely, any such set is a clique in $G \vee H$ that cannot be extended to a larger general position set of $G \vee H$. Hence $|S| \leq \omega ^-(G)+\omega ^-(H)$.
		
		Suppose now that $S \cap V(H) = \emptyset $ (the case $S \cap V(G) = \emptyset $ is symmetrical). As all vertices of $V(G)$ are at distance at most two in $G \vee H$, $S$ is in general position in $G \vee H$ if and only if it is an independent union of cliques. However, $S$ will only be a maximal general position set if it is a maximal independent union of at least two cliques in $G$ (otherwise we could extend $S$ by adding any vertex from $H$). Thus $\gp^- (G \vee H) \leq \alpha ^{\omega ^-}(G)$.
		
		The result now follows upon taking the minimum amongst these forms of maximal general position sets in $G \vee H$.
	\end{proof}		
	
	We also use the following known lemma in our comparison of the geodetic number and lower general position number. Recall that a vertex of a graph is {\em simplicial} if its neighbours induce a complete subgraph.
	
	\begin{lemma} {\rm \cite[Theorem A]{chartrand-2002}}
		\label{lem:simplicial}
		Every geodetic set of $G$ contains all simplicial vertices of $G$.
	\end{lemma}
	
	\begin{theorem}\label{thm:ab}
		Let $a\ge 2$ and $b\ge 2$ be integers. Then there is a graph $G$ with $\gp^-(G) = a$ and $\g (G) = b$ if and only if $2 \leq a \leq b$ or $4 \leq b \leq a$.
	\end{theorem}
	
	\begin{proof}
		If $a \leq b$, then by Lemma~\ref{lem:lower gp of join} the graph $bK_1 \vee K_{a-1}$ has $\gp^-(bK_1 \vee K_{a-1}) = a$. Moreover, by Lemma~\ref{lem:simplicial}, any geodetic set of $bK_1 \vee K_{a-1}$ contains the $b$ vertices of $bK_1$; since this set is geodetic, we obtain $\g (bK_1 \vee K_{a-1}) = b$. We can thus assume in the remainder of the proof that $b < a$.
		
		If $b = \g (G) = 2$, then $G$ has a universal line and hence $\gp^-(G) = 2$ by Proposition~\ref{prop:gp-=2}. Suppose that $\g (G) = 3$ and let $S = \{ x,y,z\} $ be a smallest geodetic set, so that $I[x,y] \cup I[x,z] \cup I [y,z] = V(G)$. Suppose that $S$ is not in general position, say $y$ lies on a shortest $x,z$-path $P$ in $G$. Then we have $I[x,y] \cup I[y,z] \subseteq I[x,z]$, so that $\{ x,z\} $ would be a geodetic set, a contradiction, implying that $G$ would have a maximal general position set of order at most three. Thus we cannot have $b < a$ if $b \in \{ 2,3\} $.
		
		Now we deal with the remaining case $4 \leq b < a$. Let $X_1,X_2,Y$ be cliques of orders $|X_1| = |X_2| = a-2 > 2$ and $|Y| = b-3 \geq 1$. Let $x_1 \in X_1$ and $x_2 \in X_2$ be fixed. We form a graph $H(a,b)$ from these cliques as follows. Add all possible edges between $X_1$ and $Y$, join $x_1$ to every vertex of $X_2$ and $x_2$ to every vertex of $X_1$. Add a new vertex $w$ and join it to every vertex of $X_2$. Finally we construct $G(a,b)$ by adding a vertex $z$ to $H(a,b)$ and joining it to every other vertex, i.e. $G(a,b) = H(a,b) \vee K_1$. An example is given in Fig.~\ref{fig:geodeticnum vs. lower gp}.
		
		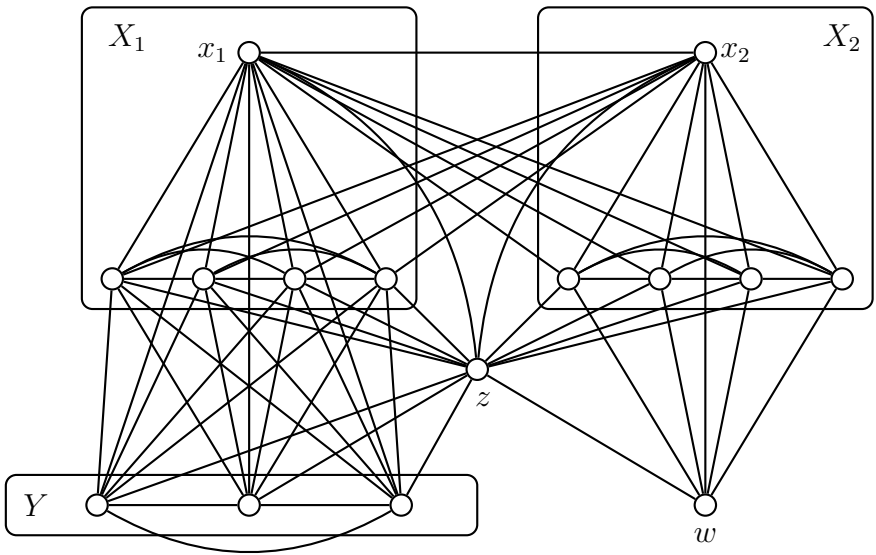
\begin{figure}[ht!]
			\centering
			\begin{tikzpicture}[x=0.4mm,y=-0.4mm,inner sep=0.2mm,scale=1.0,thick,vertex/.style={circle,draw,minimum size=8,fill=white}]

				\node at (-120,0) [vertex] (x1) {};
				\node at (-90,0) [vertex] (x2) {};
				\node at (-60,0) [vertex] (x3) {};
				\node at (-30,0) [vertex] (x4) {};
				
				\node at (120,0) [vertex] (y1) {};
				\node at (90,0) [vertex] (y2) {};
				\node at (60,0) [vertex] (y3) {};
				\node at (30,0) [vertex] (y4) {};
				
				\node at (-75,-75) [vertex] (x) {};
				\node at (75,-75) [vertex] (y) {};
				\node at (75,75) [vertex] (w) {};
				
				\node at (-125,75) [vertex] (w1) {};
				\node at (-75,75) [vertex] (w2) {};
				\node at (-25,75) [vertex] (w3) {};

				\node at (0,30) [vertex] (z) {};
				
				\path
				
				(z) edge (w)
				(z) edge (w1)
				(z) edge (w2)
				(z) edge (w3)

				(z) edge (x1)
				(z) edge (x2)
				(z) edge (x3)
				(z) edge (x4)
				
				(z) edge (y1)
				(z) edge (y2)
				(z) edge (y3)
				(z) edge (y4)
				
				(z) edge[bend right] (x)
				(z) edge[bend left] (y)
				
				(w1) edge (w2)
				(w1) edge[bend right] (w3)
				
				(w2) edge (w3)

				(w1) edge (x1)
				(w1) edge (x2)
				(w1) edge (x3)
				(w1) edge (x4)
				
				(w2) edge (x1)
				(w2) edge (x2)
				(w2) edge (x3)
				(w2) edge (x4)
				
				(w3) edge (x1)
				(w3) edge (x2)
				(w3) edge (x3)
				(w3) edge (x4)

				(w1) edge (x)
				(w2) edge (x)
				(w3) edge (x)

				(w) edge (y)
				
				(x) edge (y)
				
				(x) edge (x1)
				(x) edge (x2)
				(x) edge (x3)
				(x) edge (x4)
				(y) edge (x1)
				(y) edge (x2)
				(y) edge (x3)
				(y) edge (x4)
				
				(w) edge (y1)
				(w) edge (y2)
				(w) edge (y3)
				(w) edge (y4)
				
				(x) edge (y1)
				(x) edge (y2)
				(x) edge (y3)
				(x) edge (y4)
				(y) edge (y1)
				(y) edge (y2)
				(y) edge (y3)
				(y) edge (y4)
				(x1) edge (x2)
				(x1) edge[bend left] (x3)
				(x1) edge[bend left] (x4)
				(x2) edge (x3)
				(x2) edge[bend left] (x4)
				(x3) edge (x4)
				
				(y1) edge (y2)
				(y1) edge[bend right] (y3)
				(y1) edge[bend right] (y4)
				(y2) edge (y3)
				(y2) edge[bend right] (y4)
				(y3) edge (y4)
				;
				\draw[rounded corners] (20, 10) rectangle (130, -90) {};
				\node []  at (120, -80) {$X_2$};
				\node []  at (85, -75) {$x_2$};
				\draw[rounded corners] (-130, 10) rectangle (-20, -90) {};
				\node []  at (-115, -80) {$X_1$};
				\node []  at (-87, -75) {$x_1$};
				\draw[rounded corners] (-155, 65) rectangle (0, 85) {};
				\node []  at (-145, 75) {$Y$};
				\node []  at (2, 40) {$z$};
				\node []  at (75, 85) {$w$};
				
			\end{tikzpicture}
			\caption{The graph $G(7,6)$.}
			\label{fig:geodeticnum vs. lower gp}
		\end{figure}
		
		Let $R$ be a smallest geodetic set of $G(a,b)$. Now each vertex in $Y \cup \{ w\} $ is a simplicial vertex, so Lemma~\ref{lem:simplicial} yields $Y \cup \{ w\} \subseteq R$, and hence $\g (G(a,b)) \geq b-2$. Furthermore it is easily checked that $Y \cup \{ x_1,x_2,w\} $ is a geodetic set containing $b$ vertices, whilst adding a single vertex to $Y \cup \{ w\} $ does not yield a geodetic set; it follows that $\g (G(a,b)) = b$.
		
		For the lower general position number, Lemma~\ref{lem:lower gp of join} implies that $\gp^-(G(a,b)) = \min \{ \omega ^-(H(a,b))+1,\alpha ^{\omega ^-}(H(a,b))\} $. The set $X_2\cup \{w\}$ induces a smallest maximal clique of $H(a,b)$, hence $\omega ^-(H(a,b))+1 = a$. To calculate $\alpha ^{\omega ^-}(H(a,b)) $, let $S'$ be a smallest maximal independent union of at least two cliques. Suppose that $x_1 \in S'$. Observe that in this case $S'$ cannot contain vertices from both $(X_1 - \{ x_1\} ) \cup Y \cup \{ w\} $ and $X_2$. Hence, if $S'$ contains a vertex of $X_2$, then $S'$ would be a clique, a contradiction, so it follows that $S' \subseteq X_1 \cup Y \cup \{ w\} $; as $X_1 \cup Y \cup \{ w\} $ is an independent union of cliques, we would have $S' = X_1 \cup Y \cup \{ w\}$ and $|S'| = a+b-4 \geq a$. Thus we can assume that $x_1 \not \in S'$ and similarly $x_2 \not \in S'$. 
		
		Thus $S' \subseteq (X_1 \cup X_2 \cup Y \cup \{ w\} ) - \{ x_1,x_2\} $, from which it follows that $S' = (X_1 \cup X_2 \cup Y \cup \{ w\} ) - \{ x_1,x_2\} $ and $|S'| = 2a+b-8 > a$. Hence $\alpha ^{\omega ^-}(H(a,b)) \geq a$ and $\gp^-(G(a,b)) = a$.
	\end{proof}

	We next focus on the relationship between the lower general position number and the classical general position number. By Proposition~\ref{prop:cycle+multipartite} (ii),  for any $2 \leq a \leq b$ the complete multipartite graph with one part of cardinality $b$ and $a-1$ parts of cardinality $a$ has $\gp^-(G) = a$,  $\gp(G) = b$ and order $b + (a-1)a$. Also the graph $bK_1 \vee K_{a-1}$ from the proof of Theorem~\ref{thm:ab} has $\gp^-(bK_1 \vee K_{a-1}) = a$ and  $\gp(bK_1 \vee K_{a-1}) = b$, but a smaller order $b + a -1$. For $a = b$, trivially $K_a$ is the smallest possible such graph. We next show that for $2 \leq a < b$ a smaller graph that fulfills the same conditions can be constructed.
	
	\begin{theorem}\label{thm:gp^- vs. gp}
		For any two integers $2 \leq a < b$ there is a graph $G$ with $\gp^-(G) = a$, $\gp(G) = b$,  and $n(G) = b+\max \{ 1,a-\left \lfloor \frac{b}{2} \right \rfloor \}$.
	\end{theorem}
	
	\begin{proof}
		Let $G$ be the join of $K_c \cup K_{b-a+c}$ with $K_{a-c}$. Choose $c$ such that $a \leq b-a+2c \leq b$; then $\gp^-(G) = a$ and $\gp(G) = b$. The smallest possible choice of $c$ is $\max \{ 1,a-\left \lfloor \frac{b}{2} \right \rfloor\} $.
	\end{proof}
	Finding the smallest graph with $\gp^-(G) = a$ and $\gp(G) = b$ remains an open problem. To conclude this section we prove the following bound on the size of a graph with given lower general position number.
	
	\begin{theorem}
		If $G$ is a graph with $\gp^{-}(G) = k \geq 2$ and $n(G) \geq 2k-1$, then
		$$m(G) \le \binom{n}{2}-k+1.$$
		Moreover, the unique graph with order $n(G)$ that meets this bound is formed from a clique $K_{n(G)}$ by deleting $k-1$ edges adjacent to a fixed vertex.
	\end{theorem}
	
	\begin{proof}
		Let $n \geq 2k-1$ and take a clique $K_n$. Choose a set $S$ of $k$ vertices of $K_n$ and for some vertex $u \in S$ delete all edges between $u$ and $S-\{ u\} $ to form the graph $G(k)$. We claim that $\gp^{-}(G(k)) = k$. Let $S'$ be an arbitrary maximal general position set of $G(k)$. If $u \not \in S'$, then $S'$ must be a clique with order $n-1\geq k$, so suppose that $u \in S'$. If $(S-\{ u\} ) \cap S' = \emptyset $, then $S'$ is the clique induced by $V(G(k)) \setminus (S-\{ u\} )$, which contains $n-k+1 \geq k$ vertices. If $S'$ contains a vertex $v$ of $S- \{ u\} $, then $S' \cap V(G(k)-S) = \emptyset $, as any vertex $w \in V(G(k)-S)$ is contained in a shortest path $u,w,v$. Hence, the set $S$ is clearly a maximal general position set, so it follows that $\gp^-(G(k)) = k$.
		
		In fact we can show that this is the only graph with size $\binom{n}{2}-k+1$ that has lower general position number $k$. Assume that $G$ is any such graph and let $M$ be the set of $k-1$ edges deleted from $K_n$ to obtain $G$, i.e. $G = K_n-M$. Let $H$ be the subgraph of $K_n$ induced by $M$. Also, let $K$ be the set of vertices incident to the edges of $M$. Observe that $|K| \leq 2k-2$. If $H$ is connected, then $|K| \leq k$, and if $H$ is not connected, then there is a component of $H$ with order at most $k-1$; in either case, let $H^{\prime }$ be a component of $H$ with minimum order. Let $u$ and $v$ be two vertices of $H^{\prime }$, so that $u$ and $v$ are non-adjacent in $G$. 
		
		Let $S$ be a maximal general position set of $G$ containing $\{ u,v\} $. $S$ cannot contain any vertex outside $H^{\prime }$, so $H$ must be connected, for otherwise we have exhibited a maximal general position set of $G$ of order at most $k-1$. As $\gp ^-(G) = k$ and $|K| \leq k$, we must have that $|K| = k$, $S = K$ and $H$ is a tree. If the diameter of $H$ is at least three, then $H$ would contain a geodesic $u_0,u_1,u_2,u_3$, so that $u_0,u_3,u_1$ would be a shortest path in $G$ contained in $S$, which is impossible. We conclude that $H$ must be a star and $G \cong G(k)$.
		
		Suppose now that there is a graph $G$ with $m(G) \geq \binom{n}{2} -k+2$ and $\gp^{-}(G) = k$. Hence $G$ can be formed by deleting a set $A$ of at most $k-2$ edges from $K_n$. Let $u \sim v$ be any edge of $A$. Let $A_1$ be the set of edges of $A$ incident with $u$ (apart from $u \sim v$) and $A_2$ be the set of edges of $A$ incident with $v$ (again apart from $u \sim v$). Let $B_1$ be the set of endvertices of the edges in $A_1$ apart from $u$, and let $B_2$ be the set of endvertices of the edges of $A_2$ apart from $v$. There can be overlap between the sets $B_1$ and $B_2$, but as we have deleted at most $k-2$ edges we have $|B_1 \cup B_2 \cup \{ u,v\} | \leq k-1$. Let $S'$ be a maximal general position set of $G$ containing the set $\{ u,v\} $. $S'$ cannot contain any vertex of $V(G) - (B_1\cup B_2 \cup \{ u,v\} )$ and so $S' \subseteq B_1 \cup B_2 \cup \{ u,v\} $; this shows that $G$ contains a maximal general position set containing at most $k-1$ vertices, a contradiction.
	\end{proof}

	\section{Computational complexity}
	\label{sec:comp}
	
	In Section~\ref{sec:gp-=2} we have seen that $\gp^-(G) = 2$ if and only if $G$ has a universal line. The problem of characterising graphs that have universal lines is difficult and was extensively investigated in~\cite{rodriguez-2022}. In this section we complement these investigations by proving that the lower general position problem is NP-complete. To this end, we formally define the decision version of the problem:
	\begin{definition}\label{def:prob}
		{\sc Lower General Position} \\
		{\sc Instance}: A graph $G$, a positive integer $k\leq n(G)$. \\
		{\sc Question}: Is there a lower general position set $S$ for $G$ such that $|S|\leq k$?
	\end{definition}
	
	The problem is hard to solve, as shown by the following result.
	
	\begin{theorem}\label{thm:NP}
		The {\sc Lower General Position} problem is NP-complete.
	\end{theorem}
	
	\begin{proof}
		Let us observe that {\sc Lower General Position} is in NP since, given a set of vertices $S$, it can be tested in polynomial time if (1) $S$ is in general position, (2) if it is maximal and (3) if its cardinality is less than a given integer.
		
		We prove that {\sc Independent Dominating Set} polynomially reduces to {\sc Lower general position}. Recalling that a dominating set of a graph $G$ is a subset $V'$ of $V(G)$ such that for all $u \in V(G)\setminus V'$ there is a $v \in V'$ such that $uv\in E(G)$, we provide a formal definition of the decision version of the problem:
		\begin{quote}
			An instance of {\sc Independent Dominating Set} is given by a graph $G$ and a positive integer $k\leq n(G)$.
			The  {\sc Independent Dominating Set} problem asks whether $G$ contains a dominating set $K\subseteq V(G)$, of cardinality $k$ or less, which is also independent.
		\end{quote}
		
		The NP-completeness of {\sc Independent Dominating Set} is reported in~\cite{GarJoh}. We polynomially transform an instance $(G,k)$ of {\sc Independent Dominating Set} to an instance $(G',k')$ of {\sc Lower General Position}. In particular, given $(G,k)$ we must construct a graph $G'$ and a positive integer $k'$ such that $G$ has an independent dominating set of cardinality $k$ or less if and only if $G'$ has a lower general position set of cardinality $k'$ or less.
		
		Given an instance $(G,k)$, the graph $G'$ is built as follows:
		$$G' =(\overline{G}\cup K_{n(G)+1})\vee K_1$$
		The graph $G'$ is thus the join of graphs $G''=\overline{G}\cup K_{n(G)+1}$ and $K_1$, where $G''$ is the disjoint union of the complement of $G$ and a complete graph with $n(G)+1$ vertices.
		Then the construction of $G'$ can be achieved in polynomial time. As for $k'$, we set $k'=k+1$.
		
		We have to show that an instance $(G,k)$ of {\sc Independent Dominating Set} has a positive answer if and only if $(G',k')$, the corresponding instance of {\sc Lower General Position}, has a positive answer.
		
		By Lemma~\ref{lem:lower gp of join}, the lower general position number of $G'$ is given by $$\gp^-(G') = \min \{ \omega ^-(\overline{G}\cup K_{n(G)+1})+\omega ^-(K_1),\alpha ^{\omega ^-}(\overline{G}\cup K_{n(G)+1}),\alpha ^{\omega ^-}(K_1)\}.$$ Considering that $\omega ^-(K_1)=1$, that $\alpha ^{\omega ^-}(\overline{G}\cup K_{n(G)+1})\geq \omega ^-(\overline{G}\cup K_{n(G)+1})+1$ and that $\alpha ^{\omega ^-}(K_1)$ is not defined, we have that
		\begin{equation}\label{eq:omega}
			\gp^-(G') = \omega ^-(\overline{G}\cup K_{n(G)+1})+1= \omega ^-(\overline{G})+1.
		\end{equation}

		Moreover, observe that a smallest maximal clique in a given graph is an independent dominating set in the complement graph. Indeed, a clique is an independent set in the complement graph and, as the clique is maximal, any vertex not in the independent set of the complement graph must be adjacent to a vertex of the independent set, so the independent set is also dominating.
		
		Assume that the instance $(G,k)$ of {\sc Independent Dominating Set} has a positive answer. Then in $\overline{G}$ there is a maximal clique having cardinality at most $k$. Since $k\leq n(G)$, then  $\omega ^-(\overline{G}\cup K_{n(G)+1})\leq k$ and hence by~\eqref{eq:omega}, $\gp^-(G') = \omega ^-(\overline{G}\cup K_{n(G)+1})+1\leq k+1 = k'$, which in turn means that $(G',k')$, the instance of {\sc Lower General Position}, has a positive answer.
		
		Assume now that there is a maximal general position set  $S$ in $G'$ with cardinality $k'$ or less, that is $|S|\leq k'$. By~\eqref{eq:omega}, if $S$ contains a vertex in $K_{n(G)+1}$, then, as $|S|$ is maximal, it must contain all the vertices in $K_{n(G)+1}$ and at least one vertex in $\overline{G}\vee K_1$. So the cardinality of $S$ is at least $n(G)+2\ge k'+1$, a contradiction. Thus $S$ does not contain vertices in $K_{n(G)+1}$. It now follows from Lemma~\ref{lem:lower gp of join} that $S$ consists of the vertex of $K_1$ and a smallest maximal clique in $\overline{G}$, implying that there exists an independent dominating set in $G$ of cardinality $|S|-1\leq k$. 
	\end{proof}

	\section{Lower general position number of some families}
	\label{sec:families}
	
	In this section we determine the lower general position number of the Kneser graphs $K(n,2)$, the line graphs of complete graphs, and the Cartesian and direct products of two complete graphs.
	
	Recall that the Kneser graph $K(n,2)$ is the graph with vertex set consisting of all subsets of cardinality two (or $2$-sets for short) of the set $[n]$, with an edge between two such subsets if and only if they are disjoint.
	
	\begin{theorem}\label{thm:lowergp Kneser}
		If $n\ge 3$, then
		\[ \gp^- (K(n,2))=\begin{cases}
			3; & n \in \{3,6,7\},  \\
			4; & n \in \{5,8,9\}, \\
			5; & n \in \{10,11\}, \\
			6; & \mbox{\em otherwise}.
		\end{cases}\]
	\end{theorem}
	
	\begin{proof}
		The result is trivial for $n = 3,4$. The case $n = 5$ is the Petersen graph. Hence we can assume that $n \geq 6$. We divide our argument into three cases depending on the structure of the general position set. By the result of~\cite[Theorem 3.1]{AnaChaChaKlaTho}, any general position set is either a clique or an independent union of cliques (with a couple of additional technical properties). Let $K$ be any maximal general position set of $K(n,2)$.
		
		\medskip\noindent		
		\textbf{Case 1}: The subgraph induced by $K$ contains a clique $W$ of order at least $3$.\\
		Suppose in this case that the subgraph induced by $K$ contains a component $W'$ apart from $W$ and let $\{ a,b\} $ be a ($2$-set) vertex of $W'$. Without loss of generality, let three vertices of $W$ be $\{ 1,2\} $, $\{ 3,4\} $ and $\{ 5,6\} $.  Then any $2$-set $\{ a,b\}$ being a vertex of $W'$, must have a non-empty intersection with each of these $2$-sets $\{ 1,2\} $, $\{ 3,4\} $ and $\{ 5,6\} $, which is impossible. Thus $K$ consists of a single clique $W$, which is maximal if and only if it contains $\left \lfloor \frac{n}{2} \right \rfloor $ $2$-sets. Hence in this case the maximal general position sets have order $\left \lfloor \frac{n}{2} \right \rfloor $.
		
		\medskip\noindent		
		\textbf{Case 2}: $K$ is an independent set.\\	
		As $K(n,2)$ has diameter two, any independent set is in general position. By the Erd\H{o}s-Ko-Rado Theorem, the independence number of $K(n,2)$ is $n-1$. Let $\{ 1,2\} $ be a ($2$-set) vertex of $K$. Any set containing only vertices of the form $\{ 1,i\} $ is independent (similarly for a set containing only vertices of the form $\{ 2,i\} $), but the only maximal such set has cardinality $n-1$, as given by the Erd\H{o}s-Ko-Rado Theorem, which is greater than the maximal general position sets considered in Case 1 (unless $n = 5$, in which case $K(5,2)$ is triangle-free and an independent set of order 4 is best possible). The only other maximal independent sets of $K(n,2)$ have the structure $\{ \{ 1,2\} , \{ 1,3\} ,\{ 2,3\} \}$. However, this is not a maximal general position set, as the $2$-set $\{ 1,4\} $ can be added to make a larger general position set.
		
		\medskip\noindent		
		\textbf{Case 3}: $G[K]$ has a component isomorphic to $K_2$.\\
		Without loss of generality, assume that the clique in question is $\{ \{ 1,2\} , \{ 3,4\} \} $. Then any other vertex of $K$ must be a subset of $[4]$. The set of $6$ subsets of cardinality $2$ of $[4]$ does form a maximal general position set isomorphic to $3K_2$. In this case the order of the maximal general position set is thus 6; this is a smallest possible maximal general position set for $n \geq 12$.
	\end{proof}
	
	We continue with the line graphs of complete graphs, denoted by $L(K_n)$.
	
	\begin{theorem}
		If $n\ge 2$, then
		\[ 
		\gp^- (L(K_n)) = 
		\left\{\begin{array}{lr}
			\frac{n}{2}; & n\ {\rm even}, \\[0.15cm]
			\frac{n+3}{2}; & n\ {\rm odd.}
		\end{array}\right.
		\]
	\end{theorem}
	
	\begin{proof}
		The result is easily verified for $n \leq 4$, so we assume that $n \ge 5$. Let $S$ be a lower general position set of $L(K_n)$. It is shown in~\cite{Ghorbani-2021} that the vertices of a maximal general position set of $L(K_n)$ correspond to edges in $K_n$ that induce a disjoint union of triangles and stars containing either one or at least three edges (since any star with two edges can be completed to a triangle). Let $r$ be the number of stars in $S$. Observe that if $n$ is even, then a perfect matching in $K_n$ is a maximal general position set of $L(K_n)$, whilst if $n$ is odd the disjoint union of a triangle and a matching of cardinality $\frac{n-3}{2}$ is a maximal general position set of $L(K_n)$. This shows that $\gp^-(L(K_n)) \leq \frac{n}{2}$ if $n$ is even and $\gp^-(L(K_n)) \leq \frac{n+3}{2}$ if $n$ is odd.
		
		Suppose that $S$ does not contain any stars of $K_n$. Then either $n \equiv 0 \pmod 3$ and $S$ consists entirely of triangles, in which case $|S| = n$, or else $n \equiv 1 \pmod 3$ and $S$ corresponds to a union of $\frac{n-1}{3}$ triangles and one isolated vertex, in which case $|S| = n-1$. In either case, $|S|$ is no smaller than the claimed bounds. Hence we may assume that $S$ contains at least one star. 
		
		Suppose that there is a vertex of $K_n$ that is not covered by an edge of $S$; since this vertex could be joined to the centre of a star to form a larger general position set, this would contradict the maximality of $S$. It follows that every vertex of $K_n$ is incident with an edge in $S$ and hence $|S| = n-r$. If $n$ is even, it follows that $r \leq \frac{n}{2}$, so that $|S| \geq \frac{n}{2}$. Now suppose that $n$ is odd. If $S$ contains no triangles, then there must be a star in $S$ on an odd number of vertices of $K_n$; by the preceding discussion, such a star must be incident to at least five vertices of $K_n$, so that $S$ can contain at most $1+\frac{n-5}{2} = \frac{n-3}{2}$ stars, so that $|S| \geq \frac{n+3}{2}$. Thus we can assume that $S$ contains a triangle, in which case we again have $r \leq \frac{n-3}{2}$.
	\end{proof}
	
	We now consider the Cartesian product of two complete graphs $K_r$ and $K_s$ with $r,s \ge 2$, also known as rook graphs.
	
	\begin{theorem}
		\label{thm:K-r-K-s}
		If $r,s\ge 2$, then
		\[ \gp^- (K_r\cp K_s)=\min\{r,s\}.\]
	\end{theorem}
	
	\begin{proof}
		Since the vertices of any copy of $K_r$ or $K_s$ in $K_r\cp K_s$ form a maximal general position set, trivially $\gp^- (K_r\cp K_s)\le \min\{r,s\}$. Now, assume that there is a maximal general position set $S$ of $K_r\cp K_s$ of cardinality smaller than $\min\{r,s\}$. First note that by maximality $S$ cannot be a proper subset of the vertex set of any copy of $K_r$ or of $K_s$. Also, there must exist a copy of $K_r$, say $K_r^{(i)}$, and a copy of $K_s$, say $^{(j)}K_s$, which do not contain any vertex of $S$. This immediately allows us to observe that, independently of the structure of $S$, the unique vertex of $K_r\cp K_s$ belonging to both copies $K_r^{(i)}$ and $^{(j)}K_s$, together with the set $S$ would also form a general position set of $K_r\cp K_s$. Therefore, $S$ is not maximal, which leads to the desired equality.
	\end{proof}
	
	Theorem~\ref{thm:K-r-K-s} should be compared with~\cite[Theorem 3.2]{Ghorbani-2021}, which asserts that if $r,s\ge 2$, then $\gp(K_r \cp K_s) = r + s - 2$. 
	
	Continuing the theme of products of complete graphs, the direct product is also of interest. Given two graphs $G$ and $H$, the direct product of $G$ and $H$ is the graph $G\times H$ whose vertex set is $V(G)\times V(G)$ and two vertices $(g,h), (g',h')\in V(G\times H)$ are adjacent in $G\times H$ if and only if $gg'\in E(G)$ and $hh'\in E(H)$. For any $g \in V(G)$ we call the subgraph of $G \times H$ induced by $\{ g\} \times H$ a \emph{$H$-layer} of $G \times H$ (and similarly if $h \in V(H)$ the subgraph induced by $G \times \{ h\} $ is a \emph{$G$-layer}). Note that, unlike the case of Cartesian products, a $G$-layer will be isomorphic to $G$ only if $G$ is an empty graph.  
	
	\begin{theorem}
		\label{thm:K-r-times-K-s}
		If $r,s\ge 2$ with $(r,s)\ne (2,2)$, then
		\[ \gp^- (K_r\times K_s)=\min\{r,s,4\}.\]
	\end{theorem}
	
	\begin{proof}
		We can assume without loss of generality that $2 \leq r \leq s$ and $s \geq 3$. A subset $W$ of $V(K_r \times K_s)$ induces a clique if and only if all vertices of $W$ lie in distinct layers of $K_r \times K_s$, whereas a subset $A \subseteq V(K_r \times K_s)$ is independent if and only if it lies within a single layer. 
		
		We show first that any maximum clique $W$ is a maximal general position set. Let $(x,y) \in V(K_r \times K_s) \setminus W$; we must show that $W \cup \{ (x,y)\} $ is not in general position. If $r \geq 3$, then $W$ contains  vertices $(x,j_1)$ and $(i_2,j_2)$ such that $i_2 \neq x$ and $j_2 \neq y$; then the vertices $(x,y),(i_2,j_2),(x,j_1)$ induce a shortest path and $(x,y)$ cannot be added to $W$ to make a larger general position set. If $r = 2$, let $W = \{ (i_1,j_1),(i_2,j_2)\} $. Without loss of generality, $x = i_1$. If $y \neq j_2$, then the vertices $(i_1,j_1),(i_2,j_2),(i_1,y)$ induce a shortest path, whereas if $j_2 = y$, then for $j \in V(K_s) \setminus \{ j_1,j_2\} $ the path $(i_2,j_2),(i_1,j_1),(i_2,j),(i_1,j_2) = (x,y)$ is a geodesic. This shows that $\gp ^-(K_r \times K_s) \leq \min \{ r,s\} $. For $r = 2$, this yields $\gp ^-(K_2 \times K_s) = 2 = \min \{ 2,s,4\} $ and $K_2 \times K_s$ has a universal line, whilst for $r = 3$ we have $\gp ^-(K_3 \times K_s) \leq 3$.
		
		For $r \geq 3$ the set of vertices of any layer is a maximal general position set in $K_r \times K_s$. Suppose that $\{ i\} \times K_s$ is a $K_s$-layer (the argument for $K_r$-layers is identical). The vertices of $\{ i\} \times K_s$ lie at distance $2$ from each other, so the set is in general position, whilst if we add any vertex $(i',j)$, $i' \neq i$, then for $j_1,j_2 \in V(K_s) \setminus \{ j\} $ the path on vertices $(i,j_1),(i',j),(i,j_2)$ is a geodesic. 
		
		Therefore, as any pair of vertices either forms an independent set or a clique, for $r \geq 3$ any set of two vertices of $K_r \times K_s$ can be extended to a larger general position set and $K_r \times K_s$ does not have a universal line, so that $\gp ^-(K_r \times K_s) \geq 3$. In particular, it follows that $\gp ^-(K_3\times K_s) = 3 = \min\{ 3,s,4\} $ and we can assume that $r \geq 4$.
		
		Consider the set $S=\{(i,j),(i,j'),(i',j),(i',j')\}$ with $i\ne i'$ and $j\ne j'$. First observe that $K_r \times K_s$ has edges $(i,j) \sim (i',j')$ and $(i,j') \sim (i',j)$ and that none of the four vertices of $S$ belongs to a shortest path between two of the remaining ones. Thus, $S$ is a general position set of $K_r\times K_s$. Moreover, any other vertex $(x,y)$ of $K_r\times K_s$ not in $S$ belongs to a shortest path between two vertices of $S$ (without loss of generality $x \ne i$ and $y \ne j$, so $(x,y)$ lies on a path of length $2$ between either $(i,j)$ and $(i,j')$ or $(i,j)$ and $(i',j)$). Consequently, $S$ is a maximal general position set of $K_r\times K_s$, which leads to $3 \le \gp^- (K_r\times K_s)\le 4$.
		
		Suppose that $K_r \times K_s$ has a maximal general position set $S$ of order $3$. As $r \ge 4$, by the preceding argument if $S$ is a clique or an independent set, then $S$ could be extended to a larger set, so, since a general position set is an independent union of cliques~\cite{AnaChaChaKlaTho}, $S$ must induce a graph isomorphic to $K_1\cup K_2$ (a graph on three vertices with only one edge). If $(i,j)$ and $(i',j')$ are the vertices of the induced $K_2$ in $S$, then the vertex corresponding to the $K_1$ must be either $(i',j)$ or $(i,j')$, say $(i',j)$; however, in this case $S$ could be extended by adding $(i,j')$ to give a general position set of order $4$ with the form discussed above. Therefore for $r \ge 4$ we have $\gp ^-(K_r \times K_s) = 4 = \min\{ r,s,4\}$ and the result is proven.  
	\end{proof}

	\section{Connection with lower monophonic position number}
	\label{sec:lower mp number}
	
	In this section we relate the lower general position number to the monophonic position number mentioned in Section~\ref{sec:intro}. The monophonic position number was introduced in~\cite{thomas-2023} as follows. A path $P$ in a graph $G$ is \emph{induced} or \emph{monophonic} if $G$ contains no chords between non-consecutive vertices of $P$. A set $M \subseteq V(G)$ is in \emph{monophonic position} if no induced path in $G$ contains three or more vertices of $M$; the \emph{monophonic position number} $\mono (G)$ of $G$ is the number of vertices in a largest monophonic position set. It was shown in~\cite{thomas-2023} that for any graph $G$ we have $\mono (G) \leq \gp (G)$ and that for any $2 \leq a \leq b$ there exists a graph with $\mono (G) = a$ and $\gp (G) = b$. This question was explored further in~\cite{TuiThoCha}, which asked for the smallest possible order of a graph $G$ with $\mono (G) = a$ and $\gp (G) = b$ for given $a \leq b$.
	
	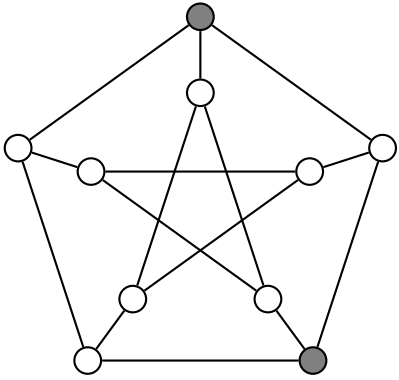
\begin{figure}[ht!]
		\centering
		\begin{tikzpicture}[x=0.2mm,y=-0.2mm,inner sep=0.2mm,scale=0.7,thick,vertex/.style={circle,draw,minimum size=10}]
			\node at (180,200) [vertex,fill=gray] (v1) {};
			\node at (8.8,324.4) [vertex,fill=white] (v2) {};
			\node at (74.2,525.6) [vertex,fill=white] (v3) {};
			\node at (285.8,525.6) [vertex,fill=gray] (v4) {};
			\node at (351.2,324.4) [vertex,fill=white] (v5) {};
			\node at (180,272) [vertex,fill=white] (v6) {};
			\node at (116.5,467.4) [vertex,fill=white] (v7) {};
			\node at (282.7,346.6) [vertex,fill=white] (v8) {};
			\node at (77.3,346.6) [vertex,fill=white] (v9) {};
			\node at (243.5,467.4) [vertex,fill=white] (v10) {};

			\path
			(v1) edge (v2)
			(v1) edge (v5)
			(v2) edge (v3)
			(v3) edge (v4)
			(v4) edge (v5)
			
			(v6) edge (v7)
			(v6) edge (v10)
			(v7) edge (v8)
			(v8) edge (v9)
			(v9) edge (v10)
			
			(v1) edge (v6)
			(v2) edge (v9)
			(v3) edge (v7)
			(v4) edge (v10)
			(v5) edge (v8)

			;
		\end{tikzpicture}
		\caption{The Petersen graph with a lower monophonic position set (grey)}
		\label{fig:Petersenmono}
	\end{figure}

	By analogy with the lower general position number, we define the {\em lower monophonic position number} $\mono^-(G)$ of $G$ to be the number of vertices in a smallest maximal monophonic position set of $G$. For any graph $G$ with order $n \geq 2$ we have $\mono ^-(G) \geq 2$ and it is easily verified that the construction of Theorem~\ref{thm:gp^- vs. gp} shows that for any $2 \leq a \leq b$ there is a graph $G$ with $\mono ^-(G) = a$ and $\mono (G) = b$. For an example of this concept, see Figure~\ref{fig:Petersenmono}, which displays a lower monophonic position set in the Petersen graph $P$. Recall that $\gp ^-(P) = 4$, so that in this case $\mono ^-(P) < \gp^- (P)$. Intuition might suggest that the relation $\mono ^-(G) \leq \gp ^-(G)$ holds generally, as with the ``ordinary'' general and monophonic position numbers. Interestingly, this turns out to be false.
	
	\begin{theorem}
		For $a,b \geq 1$, there exists a graph $G$ with $\mono^-(G) = a$ and $\gp^-(G) = b$ if and only if $a = b$, $2 \leq a < b$ or $3 \leq b < a$.
	\end{theorem}
	
	\begin{proof}
		If $a = b$ then obviously the clique $K_a$ will suffice, so we can assume that $a \neq b$. If $a$ or $b$ is one, then the graph has order one and $a = b= 1$. Also observe that if $b = 2$, then the graph has a universal line, which will also be a maximal monophonic position set, so that $a = 2$. It remains only to prove that the required graphs exist for $2 \leq a < b$ and $3 \leq b < a$.
		
		We first deal with the case $2 \leq a < b$. Take a cycle of length $6$ and identify its vertex set with $\mathbb{Z} _6$ in the natural way. Now expand each vertex $i$ into a clique $W_i$ and add all possible edges between $W_i$ and $W_{i+1}$ for $0 \leq i \leq 5$ $\pmod 6$. 
		
		Suppose that $a$ is even. Choose the cliques such that $|W_0| = |W_2| = \frac{a}{2}$ and all other cliques have order $b-\frac{a}{2}$. In Fig.~\ref{fig:a=4 and b=6} the construction for $a=4$ and $b=6$ is presented. 
		
		\begin{figure}[ht!]
			\centering
			\begin{tikzpicture}[x=0.4mm,y=-0.4mm,inner sep=0.2mm,scale=0.65,very thick,vertex/.style={circle,draw,minimum size=8,fill=white}]

				\node at (130,0) [vertex,color=red] (x00) {};
				\node at (70,0) [vertex,color=red] (x01) {};

				\node at (-70,0) [vertex] (x30) {};
				\node at (-90,0) [vertex] (x31) {};
				\node at (-110,0) [vertex] (x32) {};
				\node at (-130,0) [vertex] (x33) {};

				\node at (35,60.62) [vertex] (x10) {};
				\node at (-35,60.62) [vertex,color=red] (x20) {};
				\node at (-35,-60.62) [vertex] (x40) {};
				\node at (35,-60.62) [vertex] (x50) {};
				
				\node at (45,77.94) [vertex] (x11) {};
				
				\node at (-45,-77.94) [vertex] (x41) {};
				\node at (45,-77.94) [vertex] (x51) {};
				
				\node at (55,95.26) [vertex] (x12) {};
				
				\node at (-55,-95.26) [vertex] (x42) {};
				\node at (55,-95.26) [vertex] (x52) {};
				
				\node at (65,112.58) [vertex] (x13) {};
				\node at (-65,112.58) [vertex,color=red] (x23) {};
				\node at (-65,-112.58) [vertex] (x43) {};
				\node at (65,-112.58) [vertex] (x53) {};
				
				\path
				
				(x32) edge (x40)
				(x32) edge (x41)
				(x32) edge (x42)
				(x32) edge (x43)
				
				(x33) edge (x40)
				(x33) edge (x41)
				(x33) edge (x42)
				(x33) edge (x43)
				
				(x32) edge (x20)
				(x32) edge (x23)
				
				(x33) edge (x20)
				(x33) edge (x23)
				
				(x00) edge (x01)
				
				(x30) edge (x31)
				(x31) edge (x32)
				(x32) edge (x33)
				(x30) edge[bend left] (x32)
				(x30) edge[bend left] (x33)
				(x31) edge[bend left] (x33) 
				
				(x10) edge (x11)
				(x10) edge[bend left] (x12)
				(x10) edge[bend left] (x13)
				(x11) edge (x12)
				(x11) edge[bend left] (x13)
				(x12) edge (x13)

				(x20) edge (x23)

				(x40) edge (x41)
				(x40) edge[bend left] (x42)
				(x40) edge[bend left] (x43)
				(x41) edge (x42)
				(x41) edge[bend left] (x43)
				(x42) edge (x43)
				
				(x50) edge (x51)
				(x50) edge[bend left] (x52)
				(x50) edge[bend left] (x53)
				(x51) edge (x52)
				(x51) edge[bend left] (x53)
				(x52) edge (x53)
				
				(x00) edge (x10)
				(x00) edge (x11)
				(x00) edge (x12)
				(x00) edge (x13)
				
				(x01) edge (x10)
				(x01) edge (x11)
				(x01) edge (x12)
				(x01) edge (x13)
				
				(x00) edge (x50)
				(x00) edge (x51)
				(x00) edge (x52)
				(x00) edge (x53)
				
				(x01) edge (x50)
				(x01) edge (x51)
				(x01) edge (x52)
				(x01) edge (x53)
				
				(x10) edge (x20)
				(x10) edge (x23)
				
				(x11) edge (x20)
				(x11) edge (x23)
				
				(x12) edge (x20)
				(x12) edge (x23)
				
				(x13) edge (x20)
				(x13) edge (x23)
				
				(x50) edge (x40)
				(x50) edge (x41)
				(x50) edge (x42)
				(x50) edge (x43)
				
				(x51) edge (x40)
				(x51) edge (x41)
				(x51) edge (x42)
				(x51) edge (x43)
				
				(x52) edge (x40)
				(x52) edge (x41)
				(x52) edge (x42)
				(x52) edge (x43)
				
				(x53) edge (x40)
				(x53) edge (x41)
				(x53) edge (x42)
				(x53) edge (x43)

				(x30) edge (x40)
				(x30) edge (x41)
				(x30) edge (x42)
				(x30) edge (x43)
				
				(x31) edge (x40)
				(x31) edge (x41)
				(x31) edge (x42)
				(x31) edge (x43)
				
				(x30) edge (x20)
				(x30) edge (x23)
				
				(x31) edge (x20)
				(x31) edge (x23)

				;
			\end{tikzpicture}
			\caption{Construction for $a = 4$ and $b = 6$: the lower gp-set}
			\label{fig:a=4 and b=6}
		\end{figure}
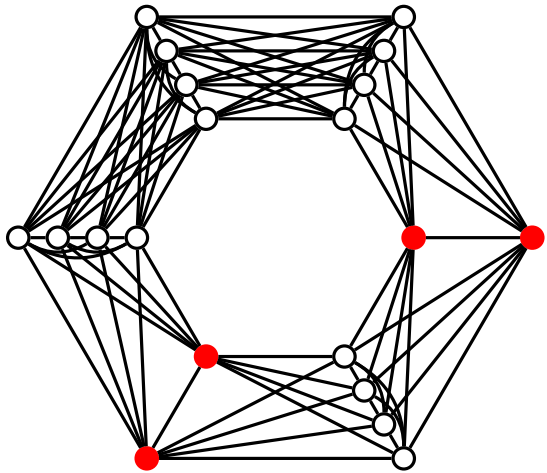
		
		As the monophonic position number of a cycle of length $6$ is $2$, it follows that the lower monophonic position number of this graph is the sum of the orders of the two smallest cliques $W_i$; as $b \geq a$ the lower monophonic position number is $a$. Similarly it follows by Proposition~\ref{prop:cycle+multipartite} that the maximal general position sets are $W_i \cup W_{i+1}$ for $0 \leq i \leq 5$ and the sum $i+1$ done mod $6$; $W_i \cup W_{i+3}$ for $0 \leq i \leq 2$; and $W_i \cup W_{i+2} \cup W_{i+4}$ for $i = 0,1$. The smallest such sets have order $b$. Therefore this graph has the required parameters. If $a$ is odd, then it can be verified that the graph with $|W_0| = \frac{a+1}{2}$, $|W_2| = \frac{a-1}{2}$ and all other parts of size $b-\frac{a-1}{2}$ works.

		Now we prove existence for $3 \leq b < a$. Define the graph $Z(w,r,s)$ as follows (see the example in Fig.~\ref{fig:lower-mp-larger-than-lower-gp}). Take a clique $W$ and divide it into four parts $W$, $R$, $S_1$ and $S_2$, where $|R| = r \geq 1$, $|S_1| = |S_2| = s \geq 1$ and $|W| = w \geq s$. The graph $Z(w,r,s)$ is formed by introducing two new vertices $x_1$ and $x_2$ and joining $x_1$ to every vertex in $S_1 \cup R$ and $x_2$ to every vertex in $S_2 \cup R$. We claim that $S = \{ x_1,x_2\} \cup S_1$ is a smallest maximal general position set of $Z(w,r,s)$. The graph $Z(w,r,s)$ has diameter two and so a subset of $V(Z(w,r,s))$ is in general position if and only if it is a disjoint union of cliques, so the set $S$ is a maximal general position set. 
		
		\begin{figure}[ht!]
			\centering
			\begin{tikzpicture}[x=0.4mm,y=-0.4mm,inner sep=0.2mm,scale=0.3,very thick,vertex/.style={circle,draw,minimum size=8,fill=white}]
				
				\node at (-300,0) [vertex] (y0) {};
				\node at (-350,86.6) [vertex] (y1) {};
				\node at (-450,86.6) [vertex,color=red] (y2) {};
				\node at (-500,0) [vertex,color=red] (y3) {};
				\node at (-450,-86.6) [vertex] (y4) {};
				\node at (-350,-86.6) [vertex] (y5) {};
				
				\node at (-200,0) [vertex,color=red] (w0) {};
				\node at (-300,-163.2) [vertex,color=red] (w5) {};

				\node at (100,0) [vertex] (x0) {};
				\node at (50,86.6) [vertex,color=red] (x1) {};
				\node at (-50,86.6) [vertex] (x2) {};
				\node at (-100,0) [vertex] (x3) {};
				\node at (-50,-86.6) [vertex] (x4) {};
				\node at (50,-86.6) [vertex] (x5) {};
				
				\node at (200,0) [vertex,color=red] (z0) {};
				\node at (100,-163.2) [vertex,color=red] (z5) {};
				\path
				
				(z0) edge (x1)
				(z0) edge (x0)
				(z0) edge (x5)
				
				(z5) edge (x0)
				(z5) edge (x5)
				(z5) edge (x4)
				
				(x0) edge (x1)
				(x1) edge (x2)
				(x2) edge (x3)
				(x3) edge (x4)
				(x4) edge (x5)
				(x5) edge (x0)
				
				(x0) edge (x2)
				(x1) edge (x3)
				(x2) edge (x4)
				(x3) edge (x5)
				(x4) edge (x0)
				(x5) edge (x1)
				
				(x0) edge (x3)
				(x1) edge (x4)
				(x2) edge (x5)

				(w0) edge (y1)
				(w0) edge (y0)
				(w0) edge (y5)
				
				(w5) edge (y0)
				(w5) edge (y5)
				(w5) edge (y4)
				
				(y0) edge (y1)
				(y1) edge (y2)
				(y2) edge (y3)
				(y3) edge (y4)
				(y4) edge (y5)
				(y5) edge (y0)
				
				(y0) edge (y2)
				(y1) edge (y3)
				(y2) edge (y4)
				(y3) edge (y5)
				(y4) edge (y0)
				(y5) edge (y1)
				
				(y0) edge (y3)
				(y1) edge (y4)
				(y2) edge (y5)
				
				;
			\end{tikzpicture}
			\caption{A graph $Z(2,2,1)$ with $\mono^{-}(G) = 4$ (left) and $\gp^-(G) = 3$  (right)}
			\label{fig:lower-mp-larger-than-lower-gp}
		\end{figure}
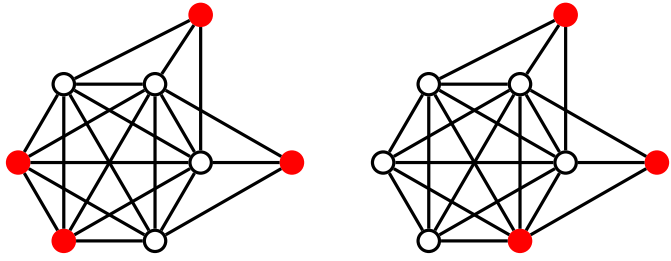
		
		Without loss of generality, apart from $\{ x_1,x_2\} \cup S_1$ there are just four types of independent unions of cliques to consider (since all the vertices within one of $W$, $S_1$, $S_2$ and $R$ are twins, it is easily seen that maximality requires taking the whole set), namely $W \cup R \cup S_1 \cup S_2$, $R\cup S_1 \cup \{ x_1\} $, $W \cup S_2 \cup \{ x_1\} $ and $W \cup \{ x_1,x_2\} $, which have orders $w+r+2s$, $r+s+1$, $w+s+1$ and $w+2$ respectively, all of which are at least $s+2$. Thus $S$ is a smallest possible maximal general position set of $Z(w,r,s)$.   
		
		Now we deal with the monophonic position sets. Using our observation on twins, we can confine our attention to the same sets as the previous paragraph; however, the set $S_1 \cup \{ x_1,x_2\} $ is not in monophonic position (if $s_1 \in S_1, s_2 \in S_2$, then $x_1,s_1,s_2,x_2$ is an induced path), leaving us with the sets $W \cup R \cup S_1 \cup S_2$, $R\cup S_1 \cup \{ x_1\} $, $W \cup S_2 \cup \{ x_1\} $ and $W \cup \{ x_1,x_2\} $, all of which are maximal monophonic position sets. The smallest of these has order $\min \{ w+2,r+s+1\} $. It follows that if $a > b \geq 3$, then the graph $Z(a-2,a-b+1,b-2)$ has the required properties.
	\end{proof}

	\section{Concluding remarks}\label{sec:conclusion}
	
	In this paper we considered general position sets of smallest cardinality that are maximal with respect to the set inclusion property. We conclude by mentioning some promising directions for future research suggested by our results.
	
	\begin{itemize}
		\item For $2 \leq a,b$, what is the smallest possible order of a graph with $\mono ^-(G) = a$ and $\gp ^-(G) = b$? What is the smallest order of a graph with $\gp ^-(G) = a$ and $\gp (G) = b$?
		\item In connection with Inequality~(\ref{eq-trivial bounds}) and the difficult problem of characterising graphs $G$ with $\gp^- (G)=2$, it could be interesting to determine further families of graphs that satisfy this property.
		\item Also in connection with Inequality~(\ref{eq-trivial bounds}), is it possible to characterise all graphs $G$ with $\gp^- (G)=\gp(G)$?
		\item Most of the graphs studied in Section~\ref{sec:families} have diameter two. The general position numbers of graphs with diameter two were determined in~\cite{AnaChaChaKlaTho}. This suggests studying lower general position sets of graphs of diameter two in general.
		\item Cartesian products with universal lines were characterised in~\cite{rodriguez-2022}. It would therefore be of interest to study the value of the lower general position numbers of Cartesian products.
		\item As discussed in Section~\ref{sec:intro}, there are several noteworthy variations of the general position number in the literature, including the mutual visibility number~\cite{DiStefano-2022}, $d$-position sets~\cite{KlaRalYer}, vertex position numbers~\cite{ThaChaTuiThoSteErs}, Steiner position numbers~\cite{KlaKuzPetYer}, edge general position numbers~\cite{manuel-2022}, mobile position sets~\cite{KlaKriTuiYer}, etc. We suggest studying lower versions of these parameters. 
	\end{itemize}

	\section*{Acknowledgements}
	Sandi Klav\v zar was supported by the Slovenian Research Agency (ARRS) under the grants P1-0297,  J1-2452, and N1-0285.
	James Tuite gratefully acknowledges funding support from EPSRC grant EP/W522338/1. Aditi Krishnakumar conducted this research with funding from an Open University research internship. I.\ G.\ Yero has been partially supported by the Spanish Ministry of Science and Innovation through the grant PID2019-105824GB-I00. The authors thank the two anonymous reviewers for their helpful suggestions. There are no datasets associated with the article.
	

\end{document}